\documentclass{amsart}

\usepackage[dvipdfmx]{graphicx}
\usepackage{color}

\usepackage{float}
\usepackage{array,booktabs}
\usepackage{amssymb}
\usepackage{subcaption}
\usepackage{amsmath}
\usepackage{subcaption}
\usepackage{bm}
\usepackage{mathrsfs}
\usepackage{overpic}
\usepackage{multimedia}
\usepackage{bbm}

\newtheorem{theorem}{Theorem}[section]
\newtheorem{lemma}[theorem]{Lemma}
\newtheorem{proposition}[theorem]{Proposition}
\theoremstyle{definition}
\newtheorem{definition}[theorem]{Definition}
\newtheorem{example}[theorem]{Example}

\newtheorem{corollary}[theorem]{Corollary}
\newtheorem{axiom}[theorem]{Axiom}
\theoremstyle{remark}
\newtheorem{remark}[theorem]{Remark}

\newcommand{\ack}{\noindent \textbf{Acknowledgements. }}
\newcommand{\org}{\noindent \textbf{Organization. }}

\DeclareMathOperator{\antipode}{s}
\DeclareMathOperator{\id}{id}
\DeclareMathOperator{\low}{low}

\newcommand{\Z}{\mathbb{Z}}
\newcommand{\N}{\mathbb{N}}
\newcommand{\R}{\mathbb{R}}
\newcommand{\hopfp}{{\rm P}}
\newcommand{\hopfP}{{\rm\bf P}}
\newcommand{\hopfq}{{\rm Q}}
\newcommand{\hopfQ}{{\rm\bf Q}}

\newcommand{\hopfh}{{\rm H}}
\newcommand{\hopfH}{{\rm\bf H}}
\newcommand{\hopfk}{{\rm K}}
\newcommand{\hopfK}{{\rm\bf K}}
\newcommand{\hopfdg}{{\rm DG}}
\newcommand{\hopfDG}{{\rm\bf DG}}

\newcommand{\hopfBF}{{\rm\bf BF}}

\newcommand{\hopfSF}{{\rm\bf SF}}
\newcommand{\hopfgp}{{\rm GP}}
\newcommand{\hopfGP}{{\rm\bf GP}}
\newcommand{\conv}{\operatorname{Conv}}
\newcommand{\cone}{\operatorname{Cone}}
\newcommand{\compo}{\rotatebox[origin=c]{180}{$\models$}}

\newcommand{\relmiddle}[1]{\mathrel{}\middle#1\mathrel{}}

\numberwithin{equation}{section}

\title{The Hopf monoid and the basic invariant of directed graphs}
\author{Keiju Kato}
\date{\today}
\email{kato.k.at@m.titech.ac.jp}
\address{Department of Mathematics, Tokyo Institute of Technology, Oh-okayama 2-12-1, Meguro-ku, Tokyo 152-8551, Japan}

\begin{document}
\maketitle

\begin{abstract}
Aguiar and Ardila defined the Hopf monoid $\hopfgp$ of generalized permutahedra and showed that it contains many submonoids that correspond to combinatorial objects. They also give a basic polynomial invariant of generalized permutahedra, which then specializes to the submonoids. We define the Hopf monoid of directed graphs and show that it also embeds in $\hopfgp$. The resulting basic invariant coincides with the strict chromatic polynomial of Awan and Bernardi.
\end{abstract}

\section{Introduction}
A Hopf monoid is an algebraic structure defined by Aguiar and Mahajan \cite{AS}. Hopf monoids may be applied in the study of combinatorial objects, in a spirit similar to earlier work \cite{JR,Joy,Schmitt,Stanley}. They provide a useful structure to many combinatorial families by matching the product to merging and the coproduct to splitting operations. On the other hand, Postnikov \cite{P}, Stanley \cite{Stanpoly} and others constructed polyhedral models to study combinatorial objects. For example, generalized permutahedra are equivalent to polymatroids and submodular functions.

In \cite{AA}, Aguiar and Ardila investigate combinatorial objects by combining these two points of view. They examine generalized permutahedra using a Hopf algebraic structure, which they call the Hopf monoid of generalized permutahedra $\hopfgp$. They also show that $\hopfgp$ contains many other Hopf monoids of combinatorial objects, such as graphs and posets. In other words, if we construct a generalized permutahedron (or a submodular function) from a combinatorial object, we investigate our combinatorial object using $\hopfgp$. As one application of this idea, there is the polynomial invariant $\chi_x(n)$. For each element $x$ of a Hopf monoid, this polynomial in $n$ is defined using a so called character $\zeta$ of the Hopf monoid. We call $\chi_x(n)$ the AA polynomial of the character $\zeta$. In many cases, the AA polynomial $\chi_x(n)$ associated to a combinatorial object $x$ is equivalent to some existing invariant. For example, we know that the AA polynomial obtained from the so called basic character of graphs is the chromatic polynomial. In particular, it satisfies Stanley's reciprocity theorem for graphs \cite{Stanrecipro}. In \cite{AA}, a reciprocity theorem is established for any AA polynomial. The reciprocity theorem is formulated in terms of the antipode of the Hopf monoid, which is analogous to the inverse in a group.

In this paper, we introduce an investigate the Hopf monoid of directed graphs. We will denote by $\hopfdg[I]$ the set of all directed graphs with vertex set $I$. We define the Hopf monoid structure for directed graphs using directed cuts. For technical reasons this has to be done so that the result is a Hopf monoid in vector species, see Section \ref{sec:pre}, and the notation changes $\hopfDG$.

Next, we define the submodular function $z_g$ on the ground set $I$ obtained from the directed graph $g$, and the generalized permutahedron $\mathcal{P}(z_g)$ obtained from $z_g$ (see Section \ref{sec:directedhopf}). We show that $\mathcal{P}(z_g)$ represents the structure of the directed graph $g$ in the following sense.

\begin{theorem}\label{thm:polytope}
For any directed graph $g\in \hopfdg[I]$ with vertex set $I$, we have
\[
\mathcal{P}(z_g)=\cone\{\,e_i-e_j\mid \mbox{ the edge }(j,i)\mbox{ is in }g\,\}\subset \R I,
\]
where $\cone$ means convex cone and for each $i\in I$, the vector $e_i$ is a standard generator of the vector space $\R I$.
\end{theorem}

The cover relation of a partially ordered set gives it a directed graph structure. In that sense, our Hopf monoid $\hopfDG$ generalizes Aguiar and Ardila's Hopf monoid $\hopfP$ of posets. In particular, Theorem \ref{thm:polytope} is a generalization of \cite[Proposition 15.1]{AA}. We prove it by an application of the max-flow min-cut theorem.

Furthermore, $\mathcal{P}(z_g)$ provides a morphism from the Hopf monoid of directed graphs to the Hopf monoid of generalized permutahedra. From Theorem \ref{thm:polytope}, we derive that the AA polynomial obtained from the basic character of directed graphs is equivalent to the strict-chromatic polynomial $\pi_g^>(n)$ \cite{AB}.

\begin{theorem}[main theorem]\label{main}
Let $\zeta$ be the character on the Hopf monoid of directed graphs $\hopfDG$ defined by
\[
\zeta(g)=
\begin{cases}
1    &(\,\text{if $g$ has no edges}\,), \\
0    &(\,\text{otherwise}\,).
\end{cases}
\]
Let $\chi_g(n)$ be the AA polynomial obtained from $\zeta$. For any directed graph $g$, we have
\[
\chi_g(n)=\pi_g^>(n).
\]
\end{theorem}

The strict-chromatic polynomial is defined by Awan and Bernardi \cite{AB} to study properties of directed graphs (see Section \ref{sec:AB}). From the reciprocity theorem of Hopf monoids, it follows that we have $(-1)^{|I|}\pi^>_g(-n)=\pi^{\geqslant}_g(n)$, where $\pi^{\geqslant}_g(n)$ is the weak-chromatic polynomial defined in \cite{AB} and $I$ is the vertex set of $g$. This fact is already established in \cite{AB}, but our proof puts it in a new context. In \cite{AB}, they also define a 3-variable polynomial invariant $B_g(n,x,y)$ for directed graphs $g$. We call this the $B$-polynomial. The strict- and weak-chromatic polynomials are specializations of the $B$-polynomial. We also find another character involving a parameter $q$ which yields $B_g(n,q,0)$ as its AA polynomial.

\org In section \ref{sec:pre}, we study some definitions and properties of Hopf monoids, as well as polynomial invariants of directed graphs. In subsections \ref{sec:prehopf}--\ref{sec:aapolygp}, we introduce further definitions and some general facts including Aguiar--Ardila's results in \cite{AA}. In subsection \ref{sec:AB}, we recall Awan--Bernardi's results in \cite{AB}. In section \ref{sec:directedhopf}, we introduce the Hopf monoid of directed graphs and we prove Theorem \ref{thm:polytope}. In section \ref{sec:directedpoly}, we introduce two characters for the Hopf monoid of directed graphs and compute the associated AA polynomials. In subsection \ref{sec:basic}, we prove Theorem \ref{main}.

\ack I should like to express my gratitude to Tam\'as K\'alm\'an for constant encouragement and much helpful advice. I should also like to thank Keita Nakagane for his useful comments and discussions.

\section{Preliminaries}\label{sec:pre}
In this section, we recall some definitions and facts about Hopf monoids that are contained in \cite{AA}. 
\subsection{Hopf monoid}\label{sec:prehopf}
First, we will introduce Joyal's notion of set species \cite{BLL,Joy}.
\begin{definition}\label{setP}
A set species $\hopfp$ satisfies the following conditions.
\begin{enumerate}
\item For each finite set $I$, a set $\hopfp[I]$ is given.
\item For each bijection $\sigma:I\to J$, there is an associated map $\hopfp[\sigma]:\hopfp[I]\to \hopfp[J]$. These satisfy $\hopfp[\sigma\circ\tau]=\hopfp[\sigma]\circ \hopfp[\tau]$ and $\hopfp[\mathrm{id}]=\mathrm{id}$.
\end{enumerate}
\end{definition}

\begin{definition}\label{morphism}
A morphism $f:\hopfp\to \hopfq$ between set species $\hopfp$ and $\hopfq$ is a collection of maps $f_I:\hopfp[I]\to \hopfq[I]$ which satisfy the following naturality axiom: for each bijection $\sigma : I \to J$, we have $f_J\circ \hopfp[\sigma]=\hopfq[\sigma]\circ f_I$. 
\end{definition}

A decomposition is a partition where the parts may be empty and are ordered. We note that the decompositions $I=S\sqcup T$ and $I=T\sqcup S$ are distinct unless $I=S=T=\emptyset$. A composition is a decomposition where no subset is empty.

\begin{definition}\label{def:hopf}
A connected Hopf monoid in set species consists of the following data.
\begin{enumerate}
\item A set species $\hopfh$ such that the set $\hopfh[\emptyset]$ is a singleton.
\item For each finite set $I$ and each decomposition $I=S\sqcup T$, product and coproduct maps
\[
\mu_{S,T}:\hopfh [S]\times \hopfh [T]\to \hopfh[I] \mbox{  and  } \Delta_{S,T}:\hopfh[I]\to \hopfh[S]\times \hopfh[T]
\]
satisfying the naturality, unitality, and compatibility axioms below.
\end{enumerate}
\end{definition}
Before stating those axioms, we discuss some notations. Fix a decomposition $I=S\sqcup T$. For $x\in \hopfh[S]$, $y\in \hopfh[T]$, and $z\in \hopfh[I]$, we write
\[
\mu_{S,T}:(x,y)\mapsto x\cdot y \mbox{  and  } \Delta_{S,T}:z\mapsto (z|_S,z/_S).
\]
We call $x\cdot y\in \hopfh[I]$ the product of $x$ and $y$. We call $z|_S \in \hopfh[S]$ the restriction of $z$ to $S$ and $z/_S \in \hopfh[T]$ the contraction of $S$ from $z$.

\begin{axiom}[Naturality]
For any decomposition $I=S\sqcup T$, any bijection $\sigma : I\to J$, and any choice of $x\in \hopfh[S]$, $y\in \hopfh[T]$, and $z\in \hopfh[I]$, we have
\[
\hopfh[\sigma](x\cdot y)=\hopfh[\sigma|_S](x)\cdot \hopfh[\sigma|_T](y),
\]
\[
\hopfh[\sigma](z)|_{\sigma(S)}=\hopfh[\sigma|_S](z|_S),\,\mbox{ and }
\hopfh[\sigma](z)/_{\sigma(S)}=\hopfh[\sigma|_T](z/_S).
\]
\end{axiom}
\begin{axiom}[Unitality]
For any finite set $I$ and any element $x\in \hopfh[I]$, we have 
\[
x\cdot 1=x=1\cdot x\mbox{ and }\ x|_I=x=x/_\emptyset,
\]
where $1$ is the unique element of $\hopfh[\emptyset]$.
\end{axiom}
\begin{axiom}[Associativity]
For any decomposition $I=R\sqcup S \sqcup T$, and any elements $x \in \hopfh[R]$, $y\in \hopfh[S]$, $z\in \hopfh[T]$, and $w \in \hopfh[I]$, we have
\[
x\cdot(y\cdot z)=(x\cdot y)\cdot z,
\] 
as well as
\[
(w|_{R\sqcup S})|_R=w|_R,\,\,(w|_{R\sqcup S})/_R=(w/_R)|_S,\,\mbox{ and }w/_{R\sqcup S}=(w/_R)/_S.
\]
\end{axiom}

\begin{axiom}[Compatibility]
Let $I=S\sqcup T=S'\sqcup T'$ be decompositions. We consider $A=S\cap S'$, $B=S\cap T'$, $C=T\sqcup S'$, and $D=T\sqcup T'$. In this situation, for any $x\in \hopfh[S]$ and $y\in \hopfh[T]$, we have
\[
(x\cdot y)|_{S'}=x|_A\cdot y|_C \mbox{  and  }
(x\cdot y)/_{S'}=x/_A\cdot y/_C .
\]
\end{axiom}

\begin{definition}
A morphism $f:\hopfh\to \hopfk$ between Hopf monoids $\hopfh$ and $\hopfk$ is a morphism of species which respects products, restrictions, and contractions; that is, we have
\[
\begin{array}{llll}
f_J(\hopfh[\sigma](x))=\hopfk[\sigma](f_I(x)) 
&\mbox{for all bijections }\sigma:I\to J \mbox{ and all } x\in \hopfh[I], \\
f_I(x\cdot y)=f_S(x)\cdot f_T(y)
&\mbox{for all }I=S\sqcup T \mbox{ and all }x \in \hopfh[S],y\in \hopfh[T], \\
f_S(z|_S)=f_I(z)|_S,\, f_T(z/_S)=f_(z)/_S
&\mbox{for all }I=S\sqcup T \mbox{ and all }z \in \hopfh[I].
\end{array}
\]
\end{definition}

Next, we introduce Hopf monoids in vector species. All vector spaces and tensor products below are over a fixed field $\mathbbm{k}$.

\begin{definition}
A vector species $\hopfP$ satisfies the following conditions.
\begin{enumerate}
\item For each finite set $I$, a vector space $\hopfP[I]$ is given.
\item To each bijection $\sigma: I\to J$, there is an associated map $\hopfP[\sigma]:\hopfP[I]\to\hopfP[J]$.
\end{enumerate}
\end{definition}

These satisfy the same axioms as in Definition \ref{setP}. A morphism of vector species $f:\hopfP\to\hopfQ$ is a collection of linear maps $f_I:\hopfP[I]\to\hopfQ[I]$ satisfying the naturality axiom of Definition \ref{morphism}.

\begin{definition}
A connected Hopf monoid in vector species is a vector species $\hopfH$ with $\hopfH[\emptyset]=\mathbbm{k}$ that is equipped with linear maps
\[
\mu_{S,T}: \hopfH[S]\otimes \mbox{\boldmath$\hopfh$}[T] \to \mbox{\boldmath$\hopfh$}[I] \mbox{  and  }\Delta_{S,T}:\mbox{\boldmath$\hopfh$}[I]\to \mbox{\boldmath$\hopfh$}[S]\otimes \mbox{\boldmath$\hopfh$}[T]
\]
for each decomposition $I=S \sqcup T$, subject to the same axioms as in Definition \ref{def:hopf}
We use similar notations as for the Hopf monoids in set species;
\[
\mu_{S,T}(x\otimes y)=x\cdot y\mbox{  and  }
\Delta_{S,T}(z)=\sum z|_S \otimes z/_S.
\]
\end{definition}

The following is a consequence of the associativity axiom. For any decomposition $I=S_1\sqcup \cdots \sqcup S_k$ with $k\ge 2$, there are unique maps
\[
\mu_{S_1,\ldots,S_k}\colon\hopfH[S_1]\otimes\cdots\otimes\hopfH[S_k]\to\hopfH[I]
\]
\[
\Delta_{S_1,\ldots,S_k}\colon\hopfH[I]\to\hopfH[S_1]\otimes\cdots\otimes\hopfH[S_k]
\]
obtained by respectively iterating the product maps $\mu$ or the coproduct maps $\Delta$ in any meaningful way. These maps are well-defined and we refer to them as the higher products and coproducts of $\hopfh$. 

Consider the linearization functor {Set} $\to$ {\bf Vec} which sends a set to the vector space with basis the given set. Applying the linearization functor to a set species $\hopfp$ gives a vector species, which we denote with $\hopfP$. If $\hopfh$ is a Hopf monoid in set species, then its linearization $\hopfH$ is a Hopf monoid in vector species. When $\hopfH$ is the linearization of a Hopf monoid $\hopfh$ in set species, then higher products and higher coproducts take the form
\[
\mu_{S_1,\ldots,S_k}(x_1\otimes \ldots \otimes x_k)=x_1\cdot\cdots\cdot x_k,
\]
\[
\Delta_{S_1,\ldots,S_k}(z)=z_1\otimes\cdots\otimes z_k
\]
whenever $x_i\in \hopfh[S_i]$ for $i=1,\ldots,k$ and $z\in\hopfh[I]$, respectively.
We refer to $z_i\in\hopfh[S_i]$ as the $i$-th minor of $z$ corresponding to the decomposition $I=S_1\sqcup\cdots\sqcup S_k$.

\subsection{The Hopf monoid of generalized permutahedra}\label{sec:genper}
In this section, we will define generalized permutahedra, which were introduced by Postnikov in \cite{P}. We remark that they are equivalent to polymatroids, which were defined earlier by Edmonds in \cite{E}.

Given a finite set $I$, let $\R I$ be the real vector space with orthogonal basis $I$. Letting $e_i$ denote the standard generator of $\R I$ for $i\in I$, we have
\[
\R I = \{ (a_i)_{i\in I} \mid a_i\in \R \}= \left\{ \sum_{i\in I} a_ie_i\relmiddle| a_i\in \R \right\}.
\]
In this section, set $|I|=n$ and let $[n]=\{1,\ldots,n\}$.
\begin{definition}
The standard permutahedron $\pi_I$ is defined by
\[
\pi_I=\conv \{ (a_i)_{i\in I} \mid \{ a_i \}_{i\in I}=[n] \},
\]
where $\conv$ means convex hull.
\end{definition}
The Minkowski sum of the polytopes $\mathfrak{p}$ and $\mathfrak{q}$ in $\R I$ is defined by
\[
\mathfrak{p}+\mathfrak{q}:=\{p+q\mid p\in\mathfrak{p},q\in\mathfrak{q}\}\subset\R I.
\]
Let $(\R I)^*$ denote the dual vector space to $\R I$. So we have
\[
(\R I)^*=\R^I:=\{ \mbox{ functions } y:I\to \R \,\},
\]
where an element $y$ of $\R^I$ is a linear functional on $\R I$. So we have 
\[
y\left( \sum_{i\in I} a_i e_i \right)= \sum_{i\in I}a_i y(i)
\mbox{ for }\sum_{i\in I}a_i e_i \in \R I.
\]
We call $y\in R^I$ a direction. Let $\{ \boldsymbol{1}_i\mid i\in I\}$ be the basis of $\R^I$ dual to the basis $\{e_i\mid i\in I\}$ of $\R I$. We write $\boldsymbol{1}_S=\sum_{i\in S}\boldsymbol{1}_i$ for each subset $S\subseteq I$. For a polyhedron $\mathfrak{p}\subset \R I$ and a direction $y\in \R^I$ in the dual space, the maximum face of $\mathfrak{p}$ in the direction of $y$ is defined by
\[
\mathfrak{p}_y:=\{ p\in\mathfrak{p} \mid y(p)\ge y(q) \mbox{ for all } q\in \mathfrak{p}\}.
\]
We call $\mathfrak{p}_y$ the $y$-maximum face of $\mathfrak{p}$. For each face $\mathfrak{q}$ of $\mathfrak{p}$, the open normal cone $\mathcal{N}_\mathfrak{p}^o(\mathfrak{q})$ and closed normal cone $\mathcal{N}_\mathfrak{p}(\mathfrak{q})$ are defined by
\[
\mathcal{N}_\mathfrak{p}^o(\mathfrak{q}):=\{ y\in\R^I \mid \mathfrak{p}_y=\mathfrak{q} \},
\]
\[
\mathcal{N}_\mathfrak{p}(\mathfrak{q})=\overline{\mathcal{N}_\mathfrak{p}^o(\mathfrak{q})}:=\{ y\in \R^I \mid \mathfrak{q} \mbox{ is a face of }\mathfrak{p}_y \},
\]
respectively. The normal fan $\mathcal{N}_{\pi_I}$ of $\mathfrak{p}\subseteq \R^I$ is the polyhedral fan consisting of the normal cones $\mathcal{N}_\mathfrak{p}(\mathfrak{q})$ for all faces $\mathfrak{q}$ of $\mathfrak{p}$.

The normal fan $\mathcal{N}_{\pi_I}$ of the standard permutahedron $\pi_I$ is the set of faces of the braid arrangement $\mathcal{B}_I$ in $\R^I$, which is the hyperplane arrangement defined by
\[
\mathcal{B}_I=\left\{ \{y\in \R^I \mid y(i)=y(j) \} \mid i,j\in I, i\ne j \right\}.
\]
For a composition $I=S_1\sqcup\cdots\sqcup S_k$, we write
\[
\mathcal{B}_{S_1\sqcup\cdots\sqcup S_k}=
\left\{ \, y\in \R^I \relmiddle|  
\begin{array}{cccc}
y(i)=y(j)    & (i,j\in S_a\,\, (\forall a)) \\
y(i)\ge y(j) & (i\in S_a, j\in S_b \,\, (a<b))
\end{array}
\right\}.
\]
A fan $\mathcal{G}$ is a coarsening of $\mathcal{F}$ if every cone of $\mathcal{F}$ is contained in a cone of $\mathcal{G}$.
\begin{definition}
A generalized permutahedron $\mathfrak{p}\subseteq \R I$ is a polyhedron whose normal fan $\mathcal{N}_{\mathfrak{p}}$ is a coarsening of the braid arrangement $\mathcal{B}_I=\mathcal{N}_{\pi_I}$ in $\R^I$.
\end{definition}
\begin{definition}
An extended generalized permutahedron $\mathfrak{p}\subseteq \R I$ is a polyhedron whose normal fan $\mathcal{N}_{\mathfrak{p}}$ is a coarsening of a subfan of the braid arrangement $\mathcal{B}_I=\mathcal{N}_{\pi_I}$ in $\R^I$.
\end{definition}
Fix the composition $I=S_1\sqcup\cdots\sqcup S_k$. For a direction $y\in \mathcal{B}_{S_1\sqcup\cdots\sqcup S_k}^\circ$, the $y$-maximal face is always the same.
So we write $\mathfrak{p}_{S_1\sqcup\cdots\sqcup S_k}:=\mathfrak{p}_y$, where $y\in \mathcal{B}_{S_1\sqcup\cdots\sqcup S_k}^\circ$.
Next we give generalized permutahedra the structure of a Hopf monoid in vector species. In order to do this, we introduce two propositions.
\begin{proposition}
Let $I=S\sqcup T$ be a decomposition. If $\mathfrak{p}\subseteq\R I$ and $\mathfrak{q}\subseteq \R T$ are bounded generalized permutahedra, then $\mathfrak{p}\times \mathfrak{q}\subseteq\R I$ is a bounded generalized permutahedron.
\end{proposition}
\begin{proposition}\label{pro:rc}\cite[Theorem 3.15]{fuji}
Let $\mathfrak{p}\subset\R I$ be a generalized permutahedron and $I=S\sqcup T$ be a decomposition. By the definition, the linear function $\boldsymbol{1}_S=\sum_{i\in S}\boldsymbol{1}_i$ is maximized at the face $\mathfrak{p}_{S,T}$ of $\mathfrak{p}$. Then there exist generalized permutahedra $\mathfrak{p}|_S\subset \R S$ and $\mathfrak{p}/_S\subset\R T$ such that
\[
\mathfrak{p}_{S,T}=\mathfrak{p}|_S \times \mathfrak{p}/_S.
\]
\end{proposition}
We call $\mathfrak{p}|_S$ and $\mathfrak{p}/_S$ the restriction and contraction of $\mathfrak{p}$ with respect to $S$, respectively.

\begin{theorem}\cite[Theorem 5.6]{AA}
Let $\hopfGP_+[I]$ be the free vector space of the extended generalized permutahedra on $I$. Define a product and a coproduct as follows.

For extended generalized permutahedra $\mathfrak{p}\in \hopfGP[S]$ and $\mathfrak{q} \in \hopfGP[T]$, their product is given by
\[
\mathfrak{p}\cdot\mathfrak{q}:=\mathfrak{p}\times\mathfrak{q}\in\hopfGP_+[I].
\]

For an extended generalized permutahedron $\mathfrak{p}\in \hopfgp_+[I]$, its coproduct with respect to $S\sqcup T$ is given by
\[
\Delta_{S,T}(\mathfrak{p})=
\; \left\{
\begin{array}{cccccccc}
\mathfrak{p}|_S\otimes\mathfrak{p}/_S &(&\mbox{if $\mathfrak{p}$ is bounded in the direction of $\boldsymbol{1}_S$} &), \\
      0         &(&\mbox{otherwise}&),
\end{array} \right.
\]
where the restriction $\mathfrak{p}|_S$ and contraction $\mathfrak{p}/_S$ are defined in Proposition \ref{pro:rc}.

These operations turn the vector species $\hopfGP_+$ into a Hopf monoid.
\end{theorem}

\subsection{Submodular functions and generalized permutahedra}\label{sec:submod}
Let $2^I$ be the power set of a finite set $I$. A Boolean function on $I$ is a function $z:2^I\to\R$ such that $z(\emptyset)=0$.

\begin{theorem}\label{thm:hopfbf}\cite[Section 12.1]{AA}
Let $\hopfBF[I]$ be the free vector space of Boolean functions on $I$. Fix a decomposition $I=S\sqcup T$. Define a product and a coproduct as follows.

If $u\in \hopfBF[S]$ and $v\in \hopfBF[T]$, define their product to be
\[
(u\cdot v)(E):=u(E\cap S)+v(E\cap T)\mbox{ for }E\subseteq I.
\]

If $z\in \hopfBF[I]$, define its coproduct to be $z|_S\otimes z/_T \in \hopfBF[S]\otimes\hopfBF[T]$, where
\[
z|_S(E):=z(E) \mbox{ for } E\subseteq S
\]
and 
\[
z/_S(E):=z(E\cup S)-z(S) \mbox{ for } E\subseteq T.
\]

These operations turn the vector species $\hopfBF$ into a Hopf monoid in vector species.
\end{theorem}

\begin{definition}
A Boolean function $z$ on $I$ is submodular if
\[
z(A\cup B)+z(A\cap B)\le z(A)+z(B)
\]
for every $A,B\subseteq I$.
\end{definition}

\begin{theorem}\cite[Theorem 12.2]{AA}
Let $\hopfSF[I]$ be the free vector space of submodular functions on $I$. Then $\hopfSF$ is a Hopf submonoid of $\hopfBF$, with the product and coproduct defined in Theorem \ref{thm:hopfbf}.
\end{theorem}
For $x\in \R I$ and $A\subseteq I$, we denote
\[
x(A)=\sum_{i\in A}x_i.
\]
\begin{definition}
The base polytope of a given Boolean function $z:2^I\to\R$ is the set 
\[
\mathcal{P}(z):=\left\{ x\in \R I \relmiddle| \sum_{i\in I}x_i=z(I)\mbox{ and }
\sum_{i\in A}x_i\le z(A)\mbox{ for all }A\subseteq I \right\}.
\]
\end{definition}

\begin{definition}
Let $z:2^I\to \R\cup \{\infty\}$ be an extended Boolean function with $z(\emptyset)=0$ and $z(I)\ne\infty$. We say that $z$ is submodular if 
\[
z(A\cup B)+z(A\cap B)\le z(A)+z(B)
\]
whenever $z(A),z(B)$ are finite.
\end{definition}
\begin{definition}\label{def:base}
The base polytope of a given extended Boolean function $z:2^I\to\R\cup\{\infty\}$ is the set 
\[
\mathcal{P}(z):=\left\{ x\in \R I \relmiddle| \sum_{i\in I}x_i=z(I)\mbox{ and }
\sum_{i\in A}x_i\le z(A)\mbox{ for all }A\subseteq I\mbox{ with }z(A)<\infty \right\}.
\]
\end{definition}
\begin{theorem}
For a polyhedron $\mathfrak{p}$ in $\R I$, the following are equivalent.
\begin{enumerate}
\item The polytope $\mathfrak{p}$ is an extended generalized permutahedron.
\item There exists an extended submodular function $z:2^I\to \R\cup\{\infty\}$ such that $\mathfrak{p}=\mathcal{P}(z)$.
\end{enumerate}
\end{theorem}

This theorem is compiled from results in \cite{fuji,P,S} by Aguiar and Ardila in \cite{AA}.

\begin{theorem}\label{sfgp}\cite[Theorem 12.7]{AA}
The collection of maps 
\[
\hopfSF_+[I]\to\hopfGP_+[I],\,\,z\mapsto\mathcal{P}(z)
\]
is an isomorphism of Hopf monoids in vector species $\hopfSF_+\cong\hopfGP_+$.
\end{theorem}

\subsection{The AA polynomial invariant of a character}\label{sec:aapoly}
\begin{definition}
Let $\hopfH$ be a connected Hopf monoid in vector species. A character $\zeta$ on $\hopfH$ is a collection of linear maps $\zeta_I: \hopfH[I]\to \mathbbm{k}$ for each finite set $I$ satisfying the following axioms.
\begin{enumerate}
\item Naturality.
For each bijection $\sigma:I\to J$ and $x\in\hopfH[I]$, we have $\zeta_J(\hopfH[\sigma](x))=\zeta_I(x)$.
\item
Multiplicativity.
For each $I=S\sqcup T$, $x\in\hopfH[S]$, and $y\in\hopfH[T]$, we have $\zeta_I(x\cdot y)=\zeta_S(x)\zeta_T(y)$.
\item
Unitality.
The map $\zeta_\emptyset:\hopfH[\emptyset]\to\mathbbm{k}$ sends $1\in \mathbbm{k}=\hopfH[\emptyset]$ to $1\in \mathbbm{k}$.
\end{enumerate}
\end{definition}

\begin{definition}\label{def:AApoly}
Let $\hopfH$ be a connected Hopf monoid and $\zeta:\hopfH \to \mathbbm{k}$ be a character of $\hopfH$. Define, for each element $x\in \hopfH[I]$ and each natural number $n\in \N$,
\[
\chi_x(n):=\sum_{I=S_1\sqcup\cdots\sqcup S_n}(\zeta_{S_1}\otimes\cdots\otimes\zeta_{S_n})\circ\Delta_{S_1,\cdots,S_n}(x),
\]
summing over all decompositions of $I$ into $n$ disjoint subsets which are allowed to be empty.
\end{definition}
\begin{remark}
For a set species $I$ and an element $x$, the function $\chi_x$ is defined on $\N$ and takes values in $\mathbbm{k}$. If we take $n=0$, we have
\[
\chi_x(0)=
\begin{cases}
\zeta_\emptyset(x)&(\,\text{if $I=\emptyset$}\,),\\
      0           &(\,\text{otherwise}\,).
\end{cases}
\]
Furthermore we note that $\chi_x(1)=\zeta_I(x)$.
\end{remark}

Recall that composition of a finite set $I\ne\emptyset$ is a decomposition $I=S_1\sqcup\cdots\sqcup S_k$ in which each subset $S_i$ is nonempty. We write composition as $I\compo(S_1,\ldots,S_k)$.

\begin{proposition}\cite[Proposition 16.1]{AA}
Let $\hopfH$ be a connected Hopf monoid in vector species, $\zeta:\hopfH\to\mathbbm{k}$ be a character, $\chi$ be defined by Definition \ref{def:AApoly}. Fix a finite set $I$ and an element $x\in\hopfH[I]$.

For each $n\in\N$, it holds that
\[
\chi_x(n)=\sum_{k=0}^{|I|}\chi_x^{(k)}\binom{n}{k}
\]
where, for each $k=0,\ldots,|I|$, we have
\[
\chi_x^{(k)}=\sum_{I\compo(T_1,\ldots,T_k)}(\zeta_{T_1}\otimes\cdots\otimes\zeta_{T_k})\circ\Delta_{T_1,\ldots,T_k}(x) \in \mathbbm{k},
\]
summing over all compositions $(T_1,\ldots,T_k)$ of $I$. Therefore $\chi_x$ is a polynomial function of $x\in \hopfH[I]$ of degree at most $|I|$.

Let $\sigma:I\to J$ be a bijection, $x\in\hopfH[I]$ and $y:=\hopfH[\sigma](x)\in\hopfH[J]$. Then we have $\chi_x=\chi_y$.
\end{proposition}
\begin{proposition}\cite[Proposition 16.3]{AA}\label{prop:eqaapoly}
Let $\hopfH$ and $\hopfK$ be two Hopf monoids in vector species. Suppose $\zeta^\hopfK$ is a character on $\hopfH$ and $\zeta^\hopfK$ is a character on $\hopfK$. We will denote by $f:\hopfH\to\hopfK$ a morphism of Hopf monoids such that
\[
\zeta^\hopfK(f(x))=\zeta^\hopfH(x)
\]
for any $I$ and $x\in\hopfH[I]$. Let $\chi^\hopfH$ and $\chi^\hopfK$ be the polynomial invariants corresponding to $\zeta^\hopfH$ and $\zeta^\hopfK$, respectively. Then
\[
\chi_{f(x)}^\hopfK=\chi_x^\hopfH
\]
for any $I$ and $x\in\hopfH[I]$.
\end{proposition}

\begin{definition}\label{def:antipode}
Let $\hopfH$ be a connected Hopf monoid in vector species. The antipode of $\hopfH$ is the collection of maps $\antipode_I:\hopfH[I]\to\hopfH[I]$ given by $\antipode_\emptyset=\id$ and, for each finite set $I\ne \emptyset$,
\[
\antipode_I= \sum_{I\scalebox{0.5}{\compo}(S_1,\ldots,S_k) \atop k\ge1}
(-1)^k\mu_{S_1,\ldots,S_k}\circ\Delta_{S_1,\ldots,S_k}.      
\]
\end{definition}
We call this equation Takeuchi's formula.

\begin{proposition}[Reciprocity for polynomial invariants]\label{RPI} \cite[Proposition 16.5]{AA}
Let $\hopfH$ be a connected Hopf monoid, $\zeta:\hopfH\to\mathbbm{k}$ be a character, and $\chi$ be the AA polynomial invariant obtained from $\zeta$. Let $\antipode$ be the antipode of $\hopfH$. Then
\[
\chi_x(-1)=\zeta_{\antipode_I(x)}.
\]
More generally, for every scalar $n\in\N$, we have
\[
\chi_x(-n)=\chi_{\antipode_I(x)}(n).
\]
\end{proposition}

\subsection{The basic character and the basic invariant of $\hopfGP$}\label{sec:aapolygp}
We introduce the basic character $\beta$ and its associated AA invariant $\chi$ on the Hopf monoid of generalized permutahedra $\hopfGP$. We call this $\chi$ the basic invariant. The property for $\hopfGP$, which we introduce in this section, also holds in $\hopfGP_+$.

\begin{definition}
The basic character $\beta$ of $\hopfGP$ is defined by
\[
\beta(\mathfrak{p})
=
\begin{cases}
1   &(\text{ if $\mathfrak{p}$ is a point }), \\
0   &(\text{ otherwise }),
\end{cases}
\]
for a generalized permutahedron $\mathfrak{p}\subset \R I$. The basic invariant $\chi$ of $\hopfGP$ is the AA polynomial obtained from the basic character $\beta$.
\end{definition}

Given a generalized permutahedron $\mathfrak{p}\subset \R I$ and a linear functional $y\in \R^I$, the generalized permutahedron $\mathfrak{p}$ is directionally generic in the direction of $y$ if the $y$-maximal face $\mathfrak{p}_y$ is a point. In this case, we will also say that $y$ is $\mathfrak{p}$-generic and that $\mathfrak{p}$ is $y$-generic.

\begin{proposition}\label{prop:generic}\cite[Proposition 17.3]{AA}
At a natural number $n$, the basic invariant $\chi$ of a generalized permutahedron $\mathfrak{p}\subset \R I$ is given by
\[
\chi_{\mathfrak{p}}(n)=(\mbox{number of $\mathfrak{p}$-generic functions $y:I\to[n]$}).
\]
\end{proposition}

\begin{proposition}\label{prop:reciprocity}\cite[Proposition 17.4]{AA}
At a negative integer $-n$, the basic invariant $\chi$ of a generalized permutahedron $\mathfrak{p}\subset \R I$ is given by 
\[
(-1)^{|I|}\chi_{\mathfrak{p}}(-n)=\sum_{y:I\to[n]}(\mbox{number of vertices of $\mathfrak{p}_y$}).
\]
\end{proposition}

Propositions \ref{prop:generic} and \ref{prop:reciprocity} were first proved in \cite{BJR}, using Stanley's combinatorial reciprocity theorem. Aguiar and Ardila give Hopf theoretic proofs of these results in \cite{AA}. 

\subsection{Awan--Bernardi's polynomial invariant for directed graphs}\label{sec:AB}
Awan and Bernardi investigate polynomial invariants for directed graphs \cite{AA}. In particular, they define the chromatic polynomial of a directed graph. A directed graph with vertex set $I$ consists of directed edges. Let us denote by $(i,j)$ the directed edge from $i\in I$ to $j\in I$. From here on, we assume that our directed graphs do not contain parallel edges. The presence of such would not change any of our constructions. We will denote by $g=(I,E)$ the directed graph $g$ with vertex set $I$ and directed edge set $E$. 

\begin{definition}[Chromatic polynomial of \cite{AB}]
Let $g=(I,E)\in\hopfdg[I]$ be a directed graph where $E$ is the directed edge set of $g$. 
The strict-chromatic polynomial $\pi_g^{>}(n)$ of $g\in\hopfdg[I]$ is defined by
\[
\pi_g^{>}(n)=|\{ f:I\to [n]\mid f(u)< f(v)\mbox{ for any } (u,v) \in E \}|.
\]
We call such functions $f:I\to [n]$ the order-preserving maps.

The weak-chromatic polynomial $\pi_g^{\geqslant}(n)$ of $g\in\hopfdg[I]$ is defined by
\[
\pi_g^{\geqslant}(n)=|\{ f:I\to [n]\mid f(u)\le f(v)\mbox{ for any } (u,v) \in E \}|.
\]
\end{definition}
Awan and Bernardi use Ehrhart theory to study these polynomials. In particular, they count lattice points in the following polytope.
\begin{definition}
Let $g=(I,E)$ be a directed graph with vertex set $I$. The ascent polytope $\Delta_g\subset \R^{I}$ of $g$ is defined by
\[
\Delta_g=\{(x_i)_{i\in I} \mid  x_i \le x_j \left( \forall(i,j)\in E\right) \}\cap [0,1]^{I}
\]
\end{definition}

\begin{proposition}
Let $g=(I,E)$ be a directed graph, and let $\Delta_g\subset \R^V$ be its ascent polytope. Then for all positive integers $n$, we have
\[
\pi^>_g(n)=|(n+1)\Delta_g^i\cap \Z^V|,\mbox{ and }\pi^{\geqslant}_g(n)=|(n-1)\Delta_g\cap \Z^V|,
\]
where $\Delta_g^i$ is the interior of $\Delta_g$ and $qP$ is the q-dilation of the region $P$.
\end{proposition}
From Ehrhart reciprocity \cite[Theorem 4.4]{CCD}, we get the following proposition.
\begin{proposition}
Let $g=(I,E)$ be an acyclic directed graph and let $n$ be a positive integer. Then we have
\[
\pi^{\geqslant}_g(-n)=(-1)^{|I|}\pi^{>}_g(n).
\]
\end{proposition}

Next, they define a three variable polynomial invariant for directed graphs. 

\begin{definition}
Let $g=(I,E)$ be a directed graph with vertex set $I$. The $B$-polynomial of $g$ is defined by 
\[
B_g(n,y,z)=\sum_{f:V\to [n]}
y^{|\{(u,v)\in E \mid f(v)>f(u)\}|}
z^{|\{(u,v)\in E \mid f(v)<f(u)\}|}.
\]
\end{definition}
The strict- and weak-chromatic polynomials are obtained from this $B$-polynomial.
\begin{proposition}
Let $g=(I,E)$ be a directed graph with vertex set $I$. Then we have
\begin{align*}
\pi_g^{>}(n)&=\mbox{the coefficients of $y^{|I|}$ in }B_g(n,y,1),
\intertext{and}
\pi_g^{\geqslant}(n)&=B_g(n,1,0).
\end{align*}
\end{proposition}

\begin{example}
Let $I=\{0,1,2\}$ be a vertex set and $g$ be the directed graph with vertex set $I$ in Figure \ref{digex1}. We get 
\[
B_g(n,y,z)=\binom{n}{1}+(2y^2+2z^2+2yz)\binom{n}{2}+(y^3+z^3+2yz(y+z))\binom{n}{3}.
\]
Moreover, we have
\[
\pi^>_g(n)=\binom{n}{3},
\]
and
\[
\pi^\geqslant_g(n)=\binom{n}{1}+2\binom{n}{2}+\binom{n}{3}.
\]
\begin{figure}[htbp]
\includegraphics[width=2.5cm]{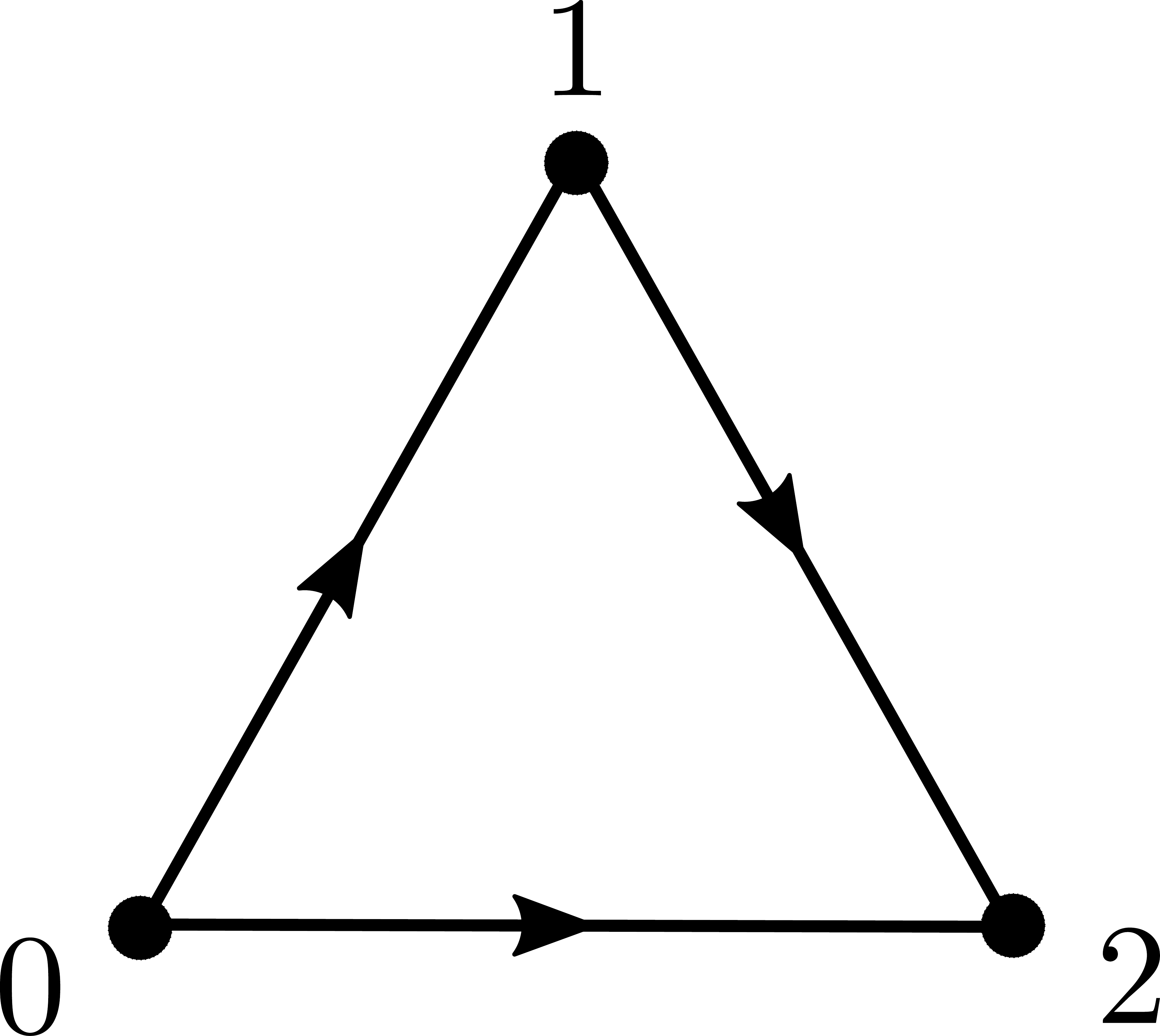}
\caption{The directed graph $g$}\label{digex1}
\end{figure} 
\end{example}

\section{The Hopf monoid of directed graphs}\label{sec:directedhopf}
Let $\hopfDG [I]$ denote the free vector space spanned by directed graphs with vertex set $I$. We use a bijection $\sigma : I\to J$ to relabel the vertices of a directed graph $g\in \hopfDG[I]$ and turn it into a graph $\hopfDG[\sigma](g)\in \hopfDG[J]$. Thus $\hopfDG$ is a vector species.

We claim that $\hopfDG$ is a Hopf monoid in vector species with the following operations. Let $I=S\sqcup T$ be a decomposition. The product $\mu_{S,T}:\hopfDG[S]\otimes \hopfDG[T]\to\hopfDG[I]$ is given by
\[
\mu_{S,T}(g_1\otimes g_2)=g_1\cdot g_2,
\]
where the graph $g_1\cdot g_2$ is the disjoint union of $g_1$ and $g_2$. So an edge of $g_1\cdot g_2$ is an edge of $g_1$ or $g_2$. The restriction $g|_S\in \hopfDG[S]$ is the induced subgraph on $S$, which consists of the edges whose ends are in $S$.

We say $S$ is a lower half of the directed graph $g$ if every directed edge which connects $S$ and $T$ is oriented from $S$. The coproduct $\Delta_{S,T}:\hopfDG[I]\to\hopfDG[S]\otimes \hopfDG[T]$ is given by
\[
\Delta_{S,T}(g)= 
\begin{cases}
g|_S\otimes g|_T &(\text{ if $S$ is a lower half of $g$ }), \\
      0          &(\text{ otherwise }).
\end{cases}
\]
We may easily check that the axioms hold.

\begin{example}
For $I=\{0,1,2,3,4\}$, let $S=\{0,1\}$ and $T=\{2,3,4\}$. With this decomposition $I=S\sqcup T$, we have, for example, 
\begin{figure}[H]
\[
\raisebox{-30pt}{\includegraphics[width=3cm]{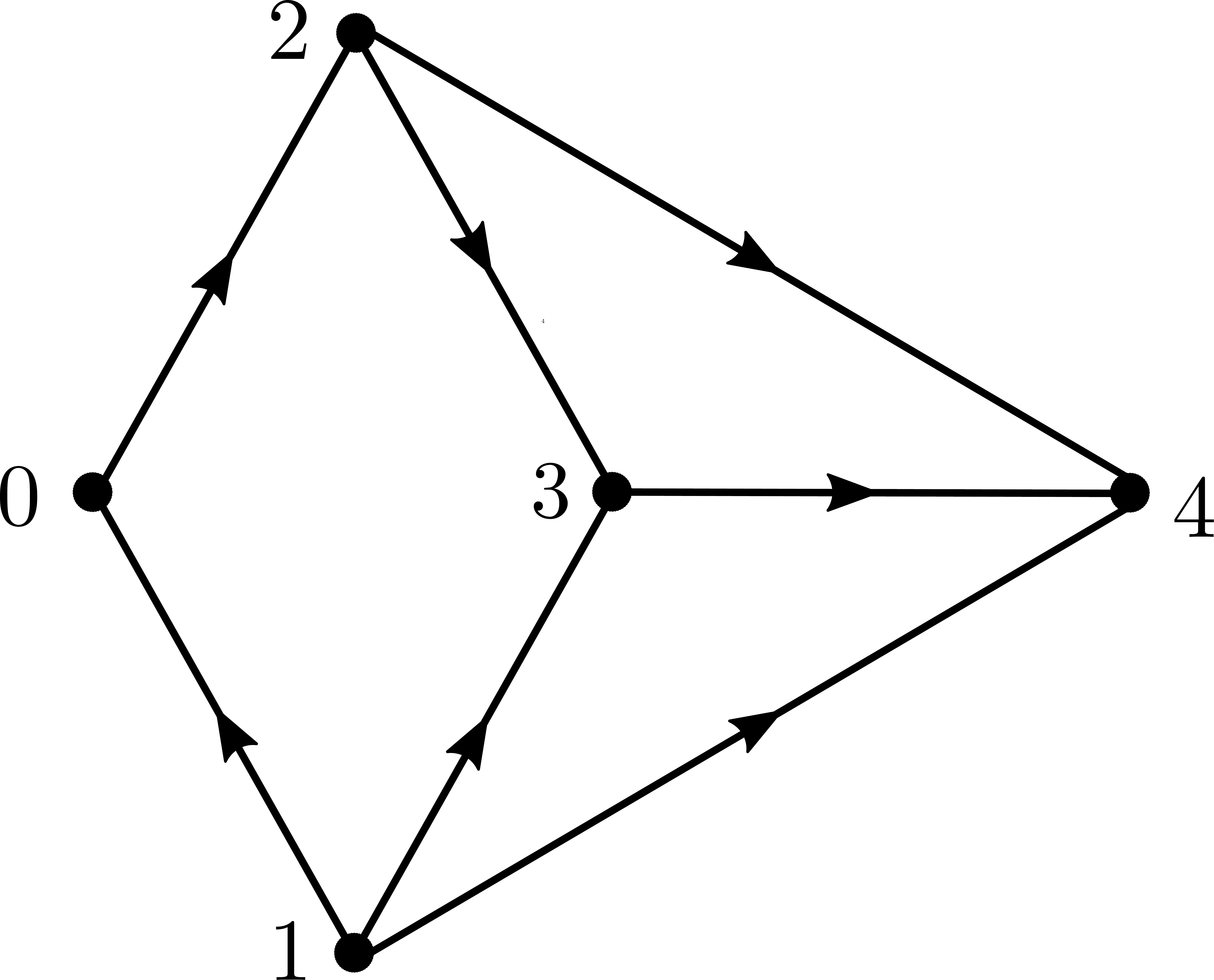}}\hspace{5pt}
\xrightarrow{\,\, \Delta_{S,T}\,\, }\hspace{5pt}
\raisebox{-15pt}{\includegraphics[width=1cm]{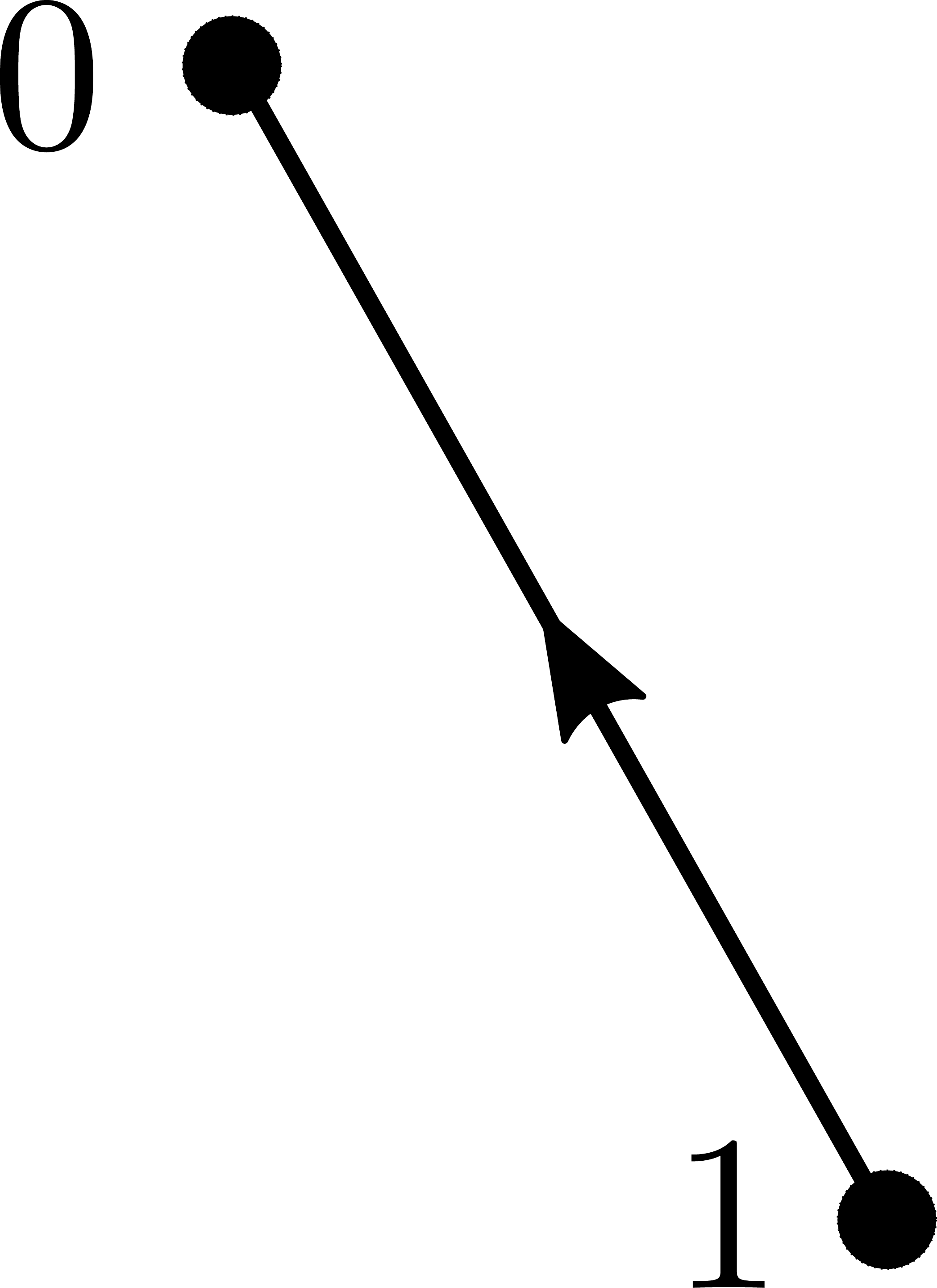}}\hspace{5pt}
\otimes\hspace{5pt}
\raisebox{-15pt}{\includegraphics[width=2.5cm]{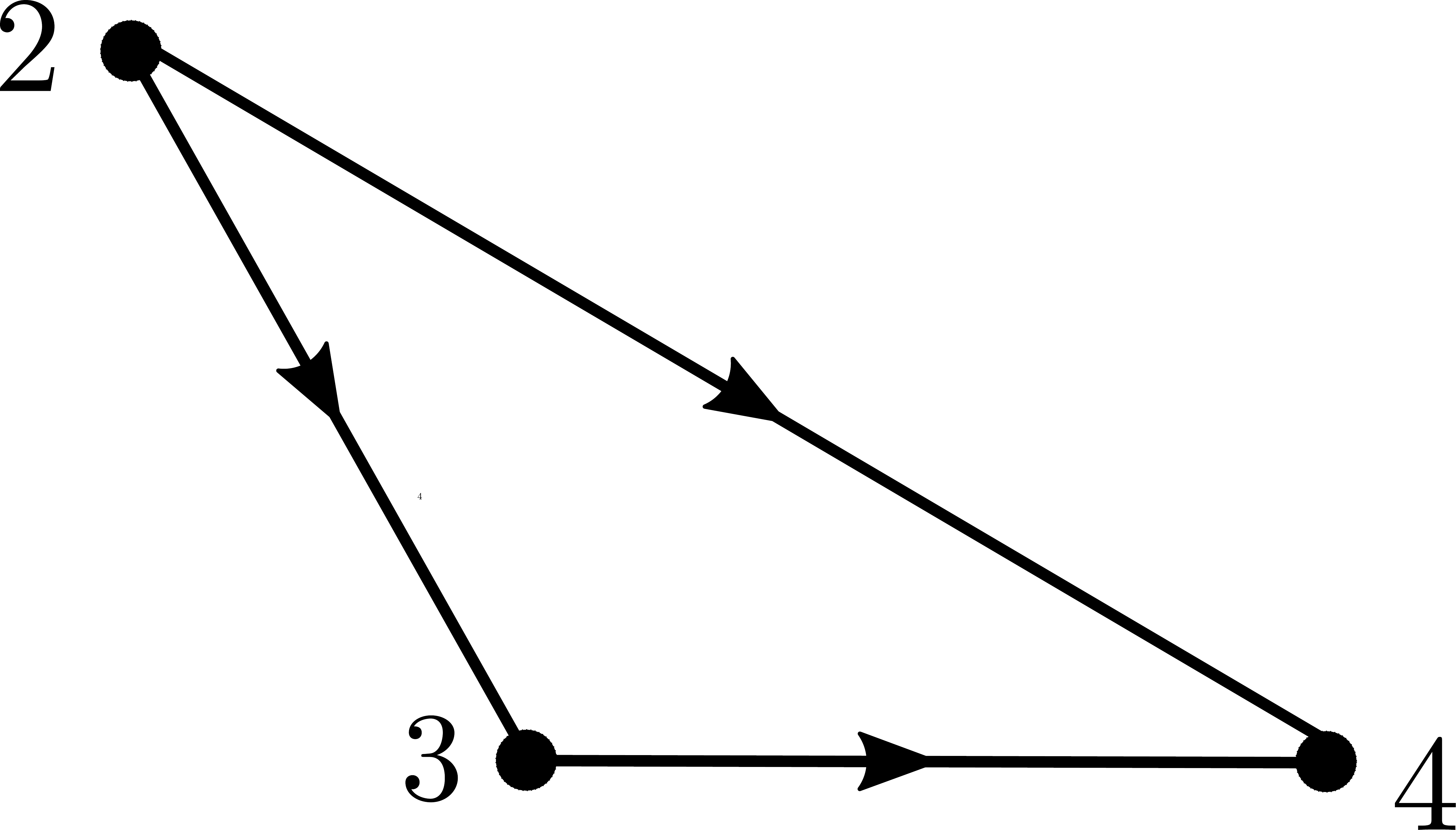}}\hspace{5pt}.
\]
\end{figure}
\end{example}

For any set $A\subset I$ and any directed graph $g\in \hopfDG[I]$, let us define the function $z_g$ by
\begin{equation}\label{eq:submodfunc}
z_g(A)= 
\begin{cases}
0            & (\text{ $A$ is a lower half of $g$ }), \\
\infty       & (\text{ otherwise }).
\end{cases}
\end{equation}
We note that $I$ is always a lower half of $g$ and hence $z_g(I)=0$.

\begin{example}
We consider the antipode of the Hopf monoid of directed graphs. We let $I={0,1,2}$ and we define $g\in \hopfDG[I]$ as in Figure \ref{digex1}. The lower halves in this directed graph are $\{0,1,2\}$, $\{0\},$ and $\{0,1\}$. So we get the antipode $\antipode_I(g)\in\hopfDG[I]$ of $g$ from Takeuchi's formula in Definition \ref{def:antipode}.
\vspace{2pt}
\[
\antipode_I(g)=-\,\raisebox{-5pt}{\includegraphics[width=0.8cm]{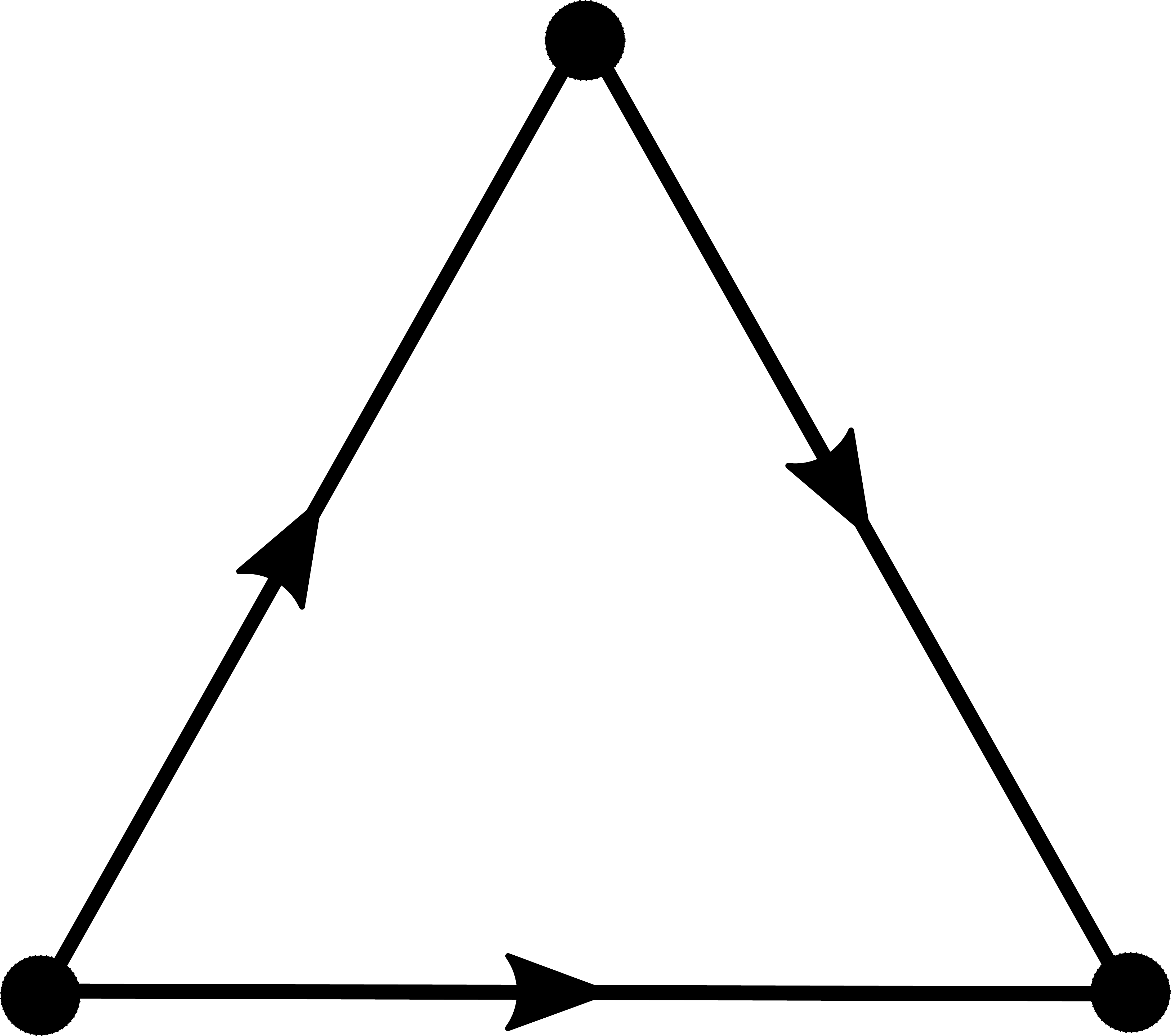}}\,
+\,\raisebox{-6pt}{\includegraphics[width=0.8cm]{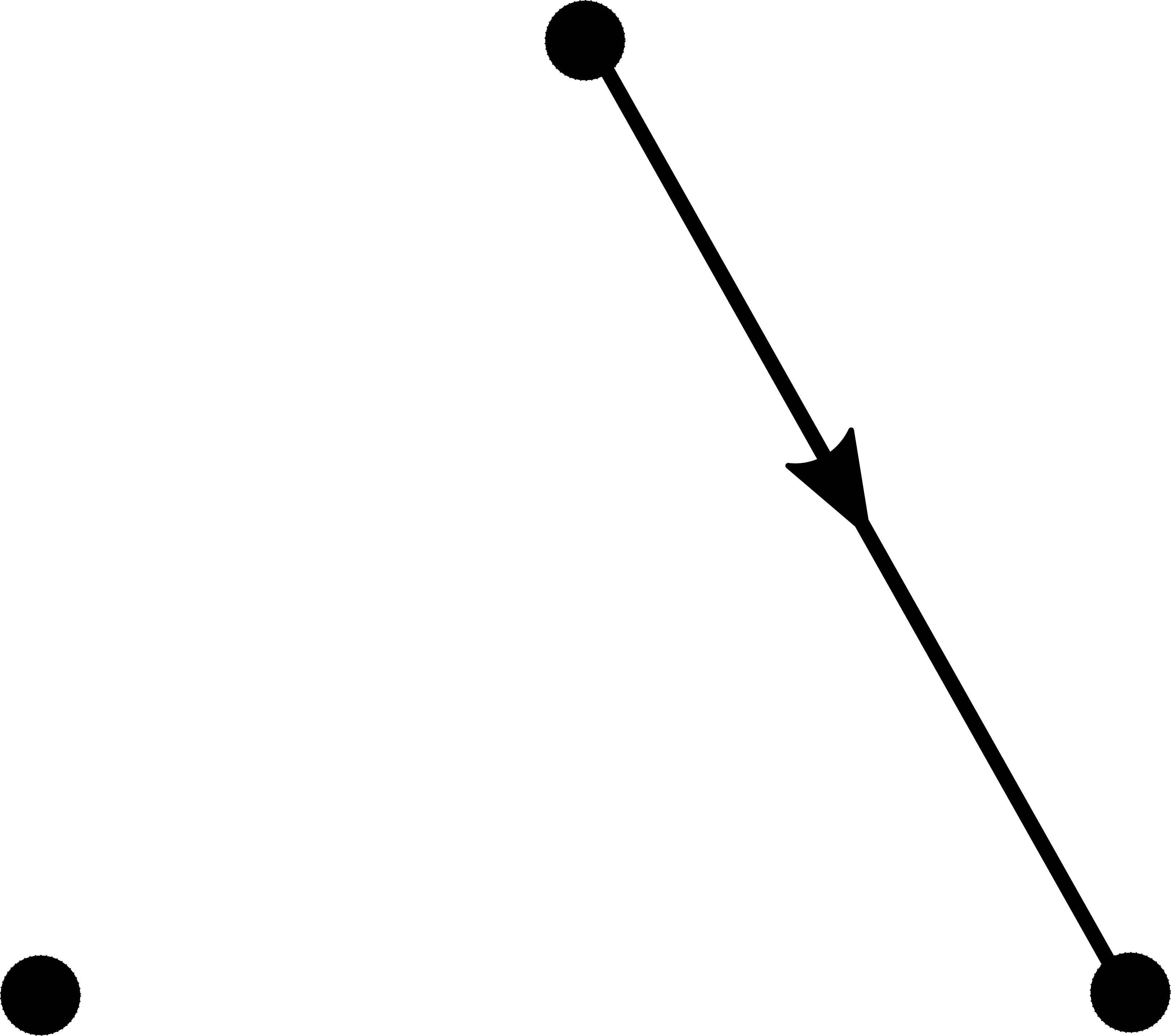}}\,
+\,\raisebox{-6pt}{\includegraphics[width=0.8cm]{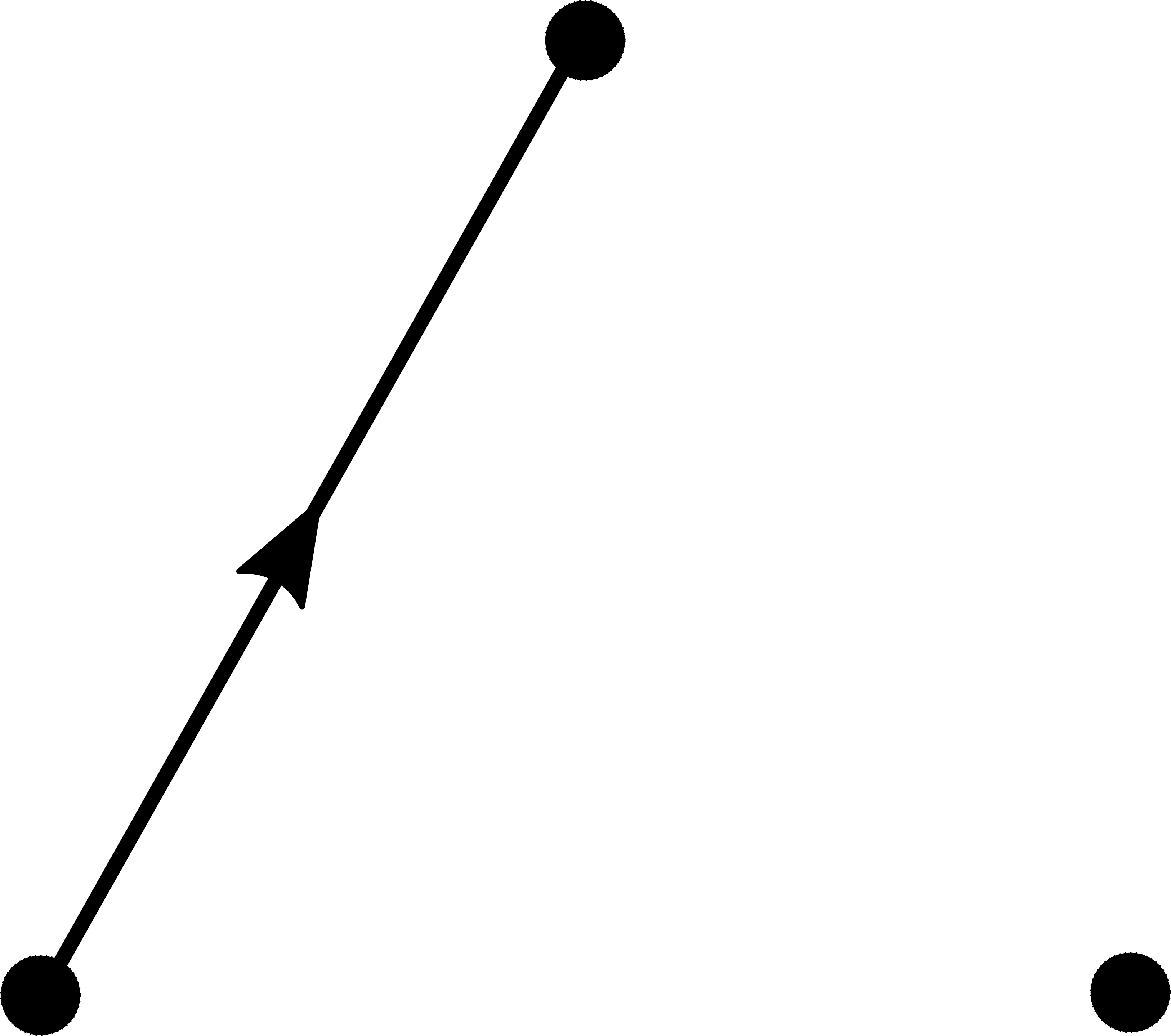}}\,
-\,\raisebox{-6pt}{\includegraphics[width=0.8cm]{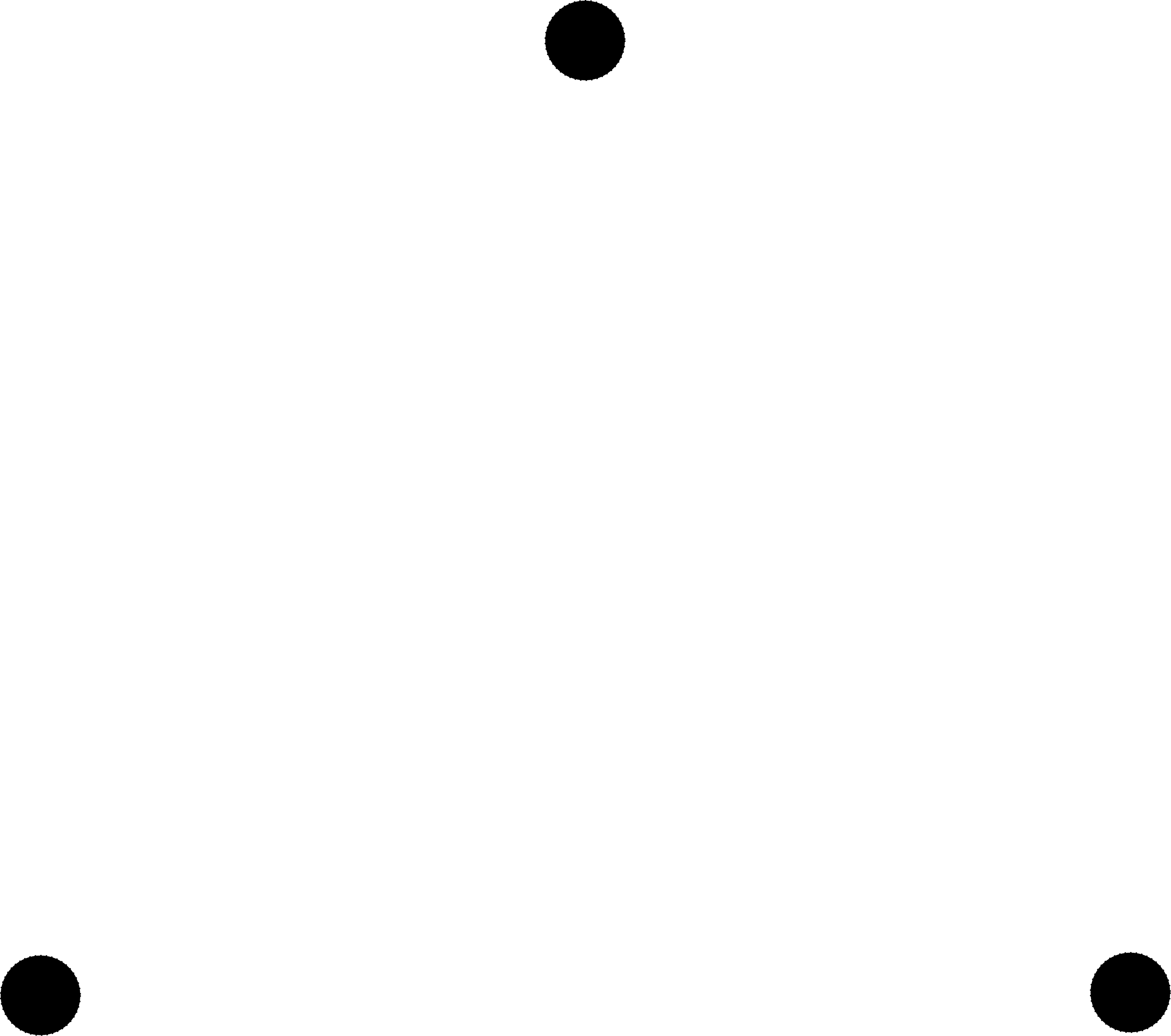}}\,\,\,\,.
\]
\vspace{2pt}

\end{example}

\begin{lemma}
For any $g\in \hopfDG[I]$, the extended Boolean function $z_g$ is submodular.
\end{lemma}
\begin{proof}
If $A,B\subseteq I$ are lower halves of $g$, then $A\cup B$ and $A\cap B$ are also lower halves of $g$. From the definition of $z_g$, if $z_g(A)=z_g(B)=0$, then the subsets $A,B$ are lower halves of $g$. Therefore if $z_g(A)=z_g(B)=0$, then we also have $z_g(A\cap B)=z_g(A\cup B)=0$. Hence $z_g$ is submodular function.
\end{proof}

Before the proof of Theorem \ref{thm:polytope}, we recall the max-flow min-cut theorem. For this, we need the notion of a flow of a directed graph. For any vertex $v$ of the directed graph $g$, we let $E_+(v)$ (resp.\ $E_-(v)$) be the set of incoming (resp.\ outgoing) edges from (resp.\ to) $v$ in $g$. We call the vertex a sink (resp.\ source) if the vertex has only incoming (resp. outgoing) edges. We also call the function $c:E\to \R_{\ge0}$ a capacity function of the directed graph $g=(I,E)$.

\begin{definition}
Let $g=(I,E)$ be a directed graph with vertex set $I$, which has a source vertex $\alpha$ and a sink vertex $\omega$. We will denote by $c:E\to \R_{\ge0}$ the capacity function of the directed graph $g$. The function $f:E\to R_{\ge 0}$ is a flow if $f$ satisfies the following conditions.
\begin{enumerate}
\item
For any edge $e\in E$, we have $f(e)\le c(e)$.
\item
For any vertex $i\in I$ except $\alpha$ and $\omega$, we have
\[
\sum_{e\in E_+(i)}f(e)=\sum_{e\in E_-(i)}f(e).
\]
\end{enumerate}
\end{definition}
For any flow $f$ of $g$, the value $[f]$ is defined by
\[
[f]=\sum_{e\in E_+(\omega)}f(e)=\sum_{e\in E_-(\alpha)}f(e).
\]

The decomposition $I=S\sqcup T$ is a cut of $g$ if the source $\alpha$ of $g$ is in $S$ and the sink $\omega$ of $g$ is in $T$. We define the capacity of the cut $I=S\cup T$ by
\[
c(S,T)=\sum_{i\in S}\sum_{j\in T}c((i,j)),
\]
where $(i,j)$ is a directed edge in $g$.

\begin{theorem}[max-flow min-cut theorem]\label{MFMC}
Let $g\in\hopfdg[I]$ be a directed graph with vertex set $I$, which has source vertex and sink vertex. Let $c:E\to \R_+$ be a capacity function of $g$. The maximal value of a flow is equal to the minimal capacity of a cut. That is, we have
\[
\max_{f:\mbox{\scriptsize flow}} [f]=\min_{I=S\sqcup T:\mbox{\scriptsize cut}}c(S,T).
\]
\end{theorem}

Now we are in a position to prove Theorem \ref{thm:polytope}.

\begin{proof}[Proof of Theorem \ref{thm:polytope}]
Let us write $\mathcal{C}(g)=\cone\{\,e_i-e_j\mid \mbox{ the edge }(j,i)\mbox{ is in }g\, \}$ for any directed graph $g=(I,E)\in \hopfdg[I]$. We will prove that, for any directed graph $g \in \hopfdg[I]$, we have $\mathcal{P}(z_g)=\mathcal{C}(g)$.

First, we will prove that $\mathcal{P}(z_g)\supset \mathcal{C}(g)$. It suffices to prove that the generators $e_i-e_j$ satisfy the conditions of Definition \ref{def:base} for $z=z_g$. For any lower half $A$ of $g$, the $a$-th coordinates ($a\in A$) of $e_i-e_j$ are equal to $0$ or $-1$. We have $(e_i-e_j)(A)\le0=z_g(A)$. We may easily check that $(e_i-e_j)(I)=0$. So we have $(e_i-e_j)\in \mathcal{P}(z_g)$. Therefore we have $\mathcal{P}(z_g)\supset \mathcal{C}(g)$.

Conversely, we will show $\mathcal{P}(z_g)\subset\mathcal{C}(g)$. We take $x\in \mathcal{P}(z_g)$. That is, for any lower half $S$ of $g$, we have $x(S)\le0$ and we have $x(I)=0$. For any $g\in \hopfdg[I]$, we will construct a new directed graph $g'$ from $g$. The vertex set of $g'$ is defined to be $I\cup\{ \alpha , \omega \}$. For any vertex $i\in I$, we add at most one new edge to form the edge set $E'$ of $g'$. If the $i$-th coordinate of $x$ is positive, we add a new edge from $i$ to $\omega$. If the $i$-th coordinate of $x$ is a negative, we add a new edge from $\alpha$ to $i$. See Figure \ref{prooffig}. For any edge $e$ of $g'$, we will define the capacity $c(e)$. If $e$ is an old edge (i.e., $e\in E$), the capacity of the edge is defined to be $c(e)=\infty$. For a new edge $e$ incident to the vertex $i\in I$, if the $i$-th coordinate of $x$ is positive, the capacity of the edge is defined to be $c(e)=x(i)$. If the $i$-th coordinate of $x$ is negative, the capacity of the edge is defined to be $c(e)=-x(i)$. Our graph $g'$ has two trivial directed cuts, formed by separating either $\alpha$ or $\omega$ from the rest of the vertices. Since $x(I)=0$, both these cuts have the same capacity
\[
\sum_{i\in I\,:\,x(i)>0}x(i)=\sum_{i\in I\,:\,x(i)<0}(-x(i)).
\]

\begin{figure}[H]
\includegraphics[width=7cm]{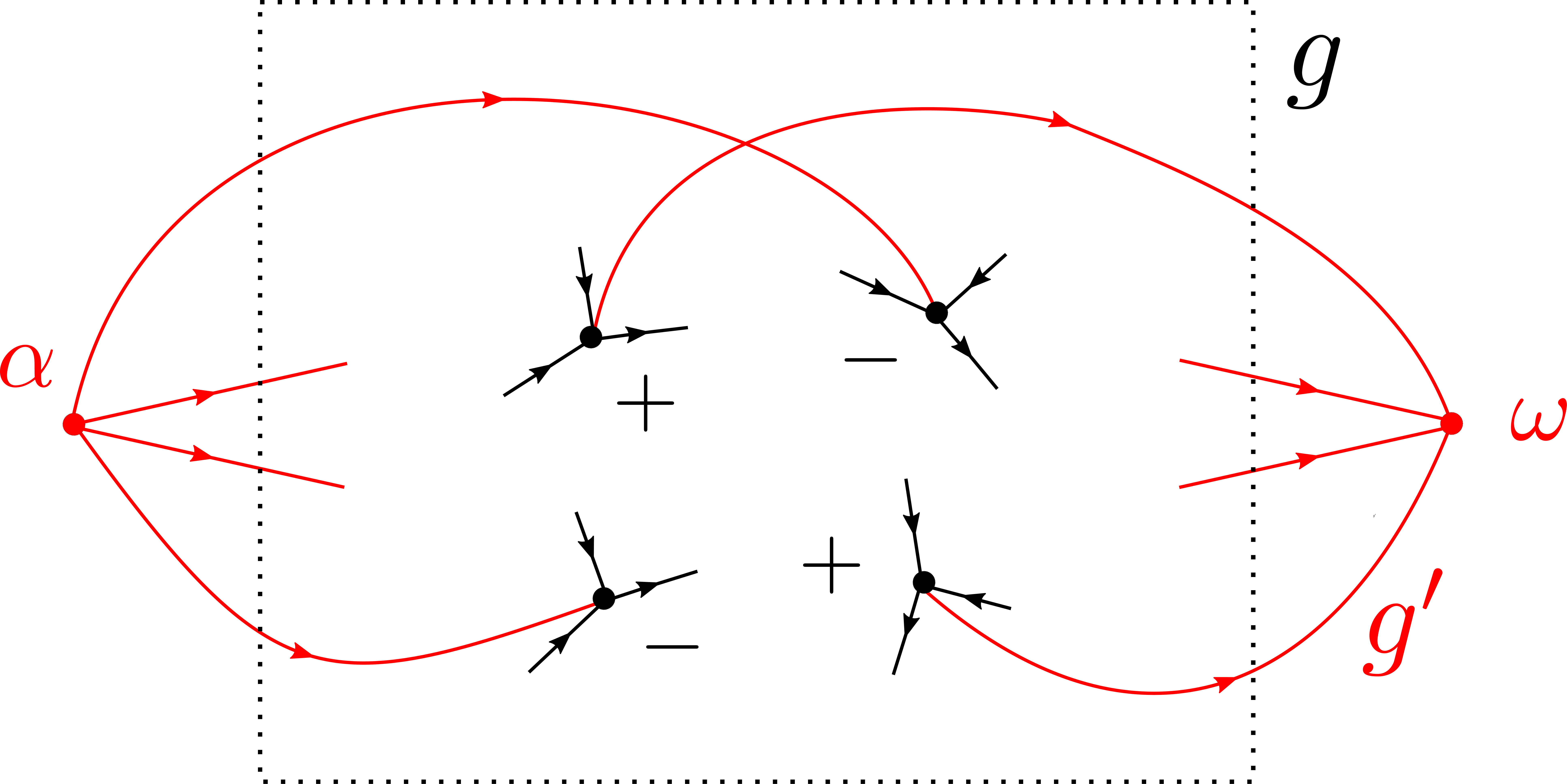}
\caption{The black graph is the old graph $g$. We construct a new graph $g'$ from $g$ by connecting each old vertex to at most one of the new vertices $\alpha$ and $\omega$.}\label{prooffig}
\end{figure}

Next, we show that the two cuts just mentioned are minimal in the directed graph $g'$ with the capacity function $c:E'\to\R_{\ge0}$. For any cut $I\cup\{\alpha,  \omega\}=A\sqcup B$ (where $\alpha\in A$ and $\omega\in B$), the sum of the capacities of the edges from $A$ to $B$ is finite if and only if there are not edges which go from $A\setminus\{\alpha\}$ to $B\setminus \{\omega\}$ in $g$. That means that $I=A\setminus\{\alpha\}\sqcup B\setminus \{\omega\}$ is a directed cut in $g$ so that $B\setminus \{\omega\}$ is its lower half. Let us denote the set of edges from $A$ to $\omega$ by $A^+$ and the set of edges from $\alpha$ to $A$ by $A^-$. We define the sets of edges $B^+,B^-$ in a similar way, see Figure \ref{notation}. Then, the capacity of the cut $I\cup\{\alpha,\omega\}=A\sqcup B$ equals
\[
\sum_{e\in A^+}c(e)+\sum_{e\in B^-}c(e).
\]
\begin{figure}[H]
\includegraphics[width=7cm]{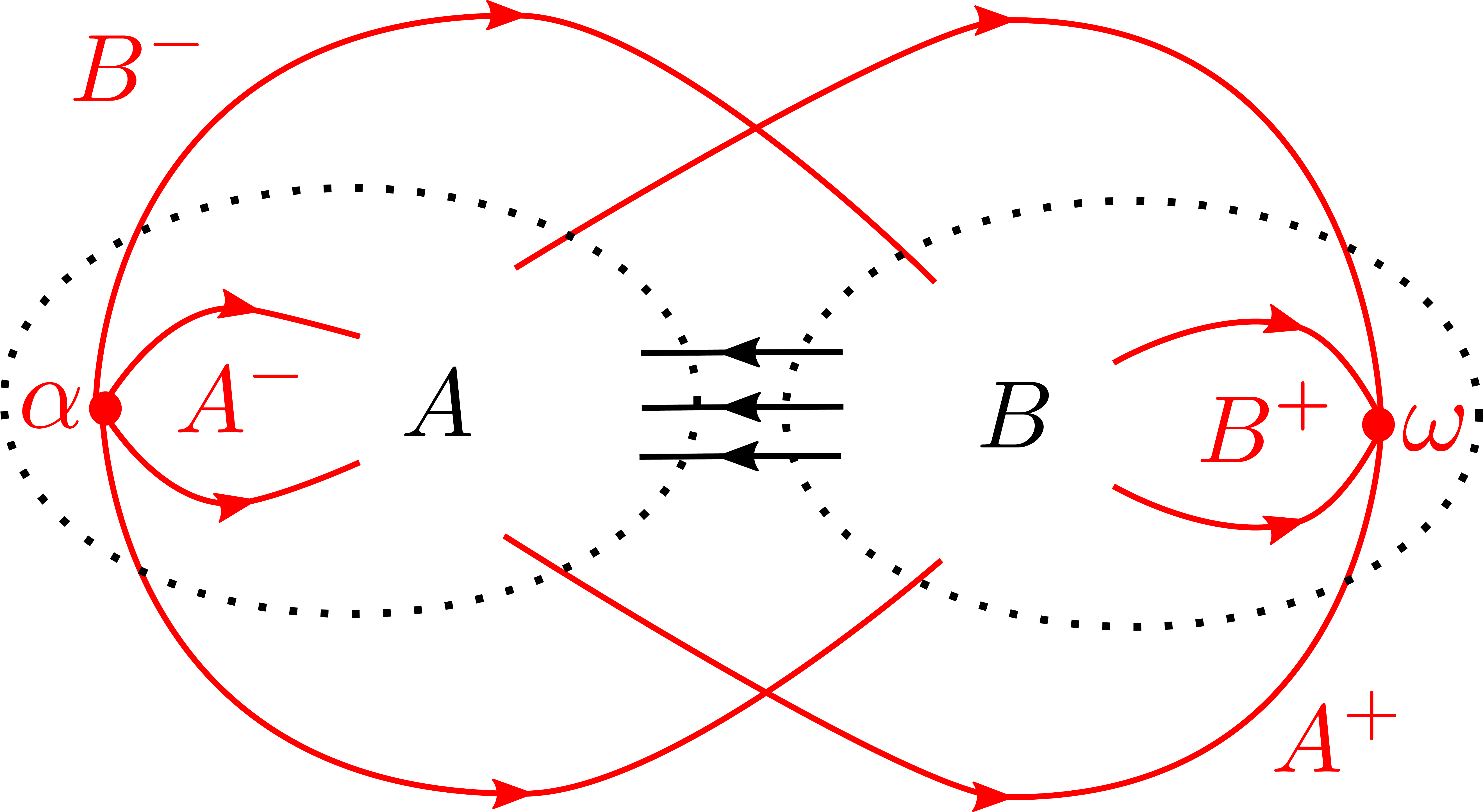}
\caption{The black regions represent the old graph $g$ and the red edges are the new edges. The left subset $B\setminus \{ \omega \}$ is a lower half of $g$.}\label{notation}
\end{figure}
Since $B\setminus \{ \omega \}$ is a lower half of $g$, we have $x(B\setminus \{ \omega \})\le0$. That implies
\[
\sum_{e\in B^+}c(e)-\sum_{e\in B^-}c(e)\le0.
\]
Therefore we get
\[
\sum_{e\in A^+}c(e)+\sum_{e\in B^-}c(e)\ge \sum_{e\in A^+}c(e)+\sum_{e\in B^+}c(e).
\]
The right hand side of this inequality is the capacity of the cut $(I\cup\{\alpha\})\cup\{\omega\}$. Hence we see that this value $M$ is indeed minimal.

Using Theorem $\ref{MFMC}$, we may take a flow $f$ on $g'$ such that
\[
[f]=\sum_{e\in E_-(\alpha)}f(e)=\sum_{e\in E_+(\omega)}f(e)=M.
\]
Notice that, for $[f]$ to reach the capacity of the trivial cuts, the value of $f$ has to be exactly the capacity on all of our new edges in $E'\setminus E$.
So, for any vertex $i\in I$ with incoming and outgoing edges $E_{\pm}(i)$, we have
\begin{eqnarray*}
\sum_{e\in E_+(i)}f(e)-\sum_{e\in E_-(i)}f(e)
&=&
\begin{cases}
f((i,\omega))    &(\text{ if }x(i)>0\,), \\
-f((\alpha,i))   &(\text{ if }x(i)<0\,),
\end{cases}\\
&=&\begin{cases}
c((i,\omega))    &(\text{ if }x(i)>0\,), \\
-c((\alpha,i))   &(\text{ if }x(i)<0\,),
\end{cases}\\
&=&x(i)
\end{eqnarray*}
That is, if we take $\lambda_{ij}=f((j,i))\ge0$ for any edge $(i,j)$, we have $x=\sum\lambda_{ij}(e_i-e_j)$. Hence $x\in \mathcal{C}(g)$, i.e., we have $\mathcal{P}(z_g)\subset\mathcal{C}(g)$.

Therefore we have $\mathcal{P}(z_g)=\mathcal{C}(g)$, which proves the theorem.
\end{proof}

We will see in the next section how the proof of the main theorem relies on Theorem \ref{thm:polytope}. But before that, we need another essential ingredient which is an extension of \cite[Proposition 15.6]{AA}.

Recall that the Hopf monoid $\hopfDG[I]$ in vector species of directed graphs is the free vector space spanned by directed graphs with vertex set $I$. For a directed graph $g\in \hopfDG[I]$, let us rename the extended submodular function $z_g$ of (\ref{eq:submodfunc}) as $\low_g$. That is, for a directed graph $g$, $\low_g:2^I\to\R\cup\{\infty\}$ is defined by
\[
\low_g(S)=
\begin{cases}
0         &(\text{ if $S$ is a lower half of $g$ }), \\
\infty    &(\text{ otherwise }).
\end{cases}
\]

\begin{proposition}\label{pro:mor}
The map $\low:\hopfDG\to\hopfSF_+$ is a morphism of Hopf monoids in vector species.
\end{proposition}

\begin{proof}
First, we examine products. Let $I=S\sqcup T$ be a decomposition. Let $g_1\in \hopfDG[S]$, $g_2\in \hopfDG[T]$ be directed graphs. We denote by $g_1\cdot g_2$ the disjoint union of $g_1$ and $g_2$. For any subset $J\subseteq I$, the subset $J$ is a lower half of $g_1\cdot g_2$ if and only if $J\cap S$ is a lower half of $g_1$ and $J\cap T$ is a lower half of $g_2$. Therefore we have
\begin{eqnarray*}
\low_{g_1\cdot g_2}(J)
&=&\low_{g_1}(J\cap S)+\low_{g_2}(J\cap T)\\
&=&(\low_{g_1}\cdot \low_{g_2})(J).
\end{eqnarray*}

Next, we look at coproduct in the Hopf monoid. For a decomposition $I=S \sqcup T$ and a directed graph $g\in\hopfDG[I]$, we will prove that $\Delta_{S,T}(\low_g)=\low_{\Delta_{S,T}(g)}$. I.e., we have $(\low_g)|_S =(\low_{g|_S})$ and $(\low_g)/_S =(\low_{g/_S})$.

Suppose $S$ is not a lower half of $g$. Then we have $\Delta_{S,T}(g)=0$ and thus $\low_{\Delta_{S,T}(g)}=0$. On the other hand, we have $\low_g(S)=\infty$. From the definition of the coproduct of submodular functions, we obtain $\Delta_{S,T}(\low_g)=0$. Therefore we get $\Delta_{S,T}(\low_g)=\low_{\Delta_{S,T}(g)}$.

Suppose $S$ is a lower half of $g$. Then we have $\low_g(S)=0$. To see that $\low$ is compatible with restriction, we note that, for any $R\subseteq S$,
\[
\low_{g|_S}(R)
=
\begin{cases}
0         &(\text{ if $R$ is a lower half of $g|_S$ }), \\
\infty    &(\text{ otherwise }),
\end{cases}
\]
and we have
\begin{eqnarray*}
(\low_{g})|_S(R)
&=&\low_g(R)\\
&=&
\begin{cases}
0         &(\text{ if $R$ is a lower half of $g$ }), \\
\infty    &(\text{ otherwise }).
\end{cases}
\end{eqnarray*}
Since $R$ is a lower half of $g|_S$ if and only if $R$ is a lower half of $g$, we have $\low_{g|_S}=(\low_g)|_S$.

To see that $\low$ is compatible with contraction, we note that, for any $R\subseteq T$,
\begin{eqnarray*}
\low_{g/_S}(R)
&=&\low_{g|_T}(R)\\
&=&\low_g(R)\\
&=&
\begin{cases}
0         &(\text{ if $R$ is a lower half of $g|_T$ }), \\
\infty    &(\text{ otherwise }),
\end{cases}
\end{eqnarray*}
and we have
\begin{eqnarray*}
(\low_{g})/_S(R)
&=&\low_g(R\sqcup S)-\low_g(S)\\
&=&\low_g(R \sqcup S)\\
&=&
\begin{cases}
0         &(\text{ if $R\sqcup S$ is a lower half of $g$ }), \\
\infty    &(\text{ otherwise }).
\end{cases}
\end{eqnarray*}
Since $R$ is a lower half of $g|_T$ if and only if $R\sqcup S$ is a lower half of $g$, we have $\low_{g/_S}=(\low_g)/_S$.

Therefore, we conclude that $\low$ is a morphism of Hopf monoids.
\end{proof}

\section{Polynomial invariants of directed graphs from characters}\label{sec:directedpoly}
In this section, we introduce two characters and their associated AA polynomials $\chi$. Moreover we obtain combinatorial formulae for $\chi(n)$ and $\chi(-n)$ for $n\in \N$.

\subsection{Basic invariant}\label{sec:basic}
First, we introduce the basic character $\beta$ of the Hopf monoid $\hopfDG$ and its associated AA polynomial $\chi(x)$, called basic invariant.
\begin{definition}
The basic character $\zeta$ of $\hopfDG$ is given by
\[
\zeta(g)=
\begin{cases}
1    &(\,\text{if $g$ has no edges}\,), \\
0    &(\,\text{otherwise}\,).
\end{cases}
\]
for a directed graph $g \in \hopfDG[I]$. The basic invariant $\chi$ of $\hopfDG$ is the AA polynomial obtained from $\zeta$.
\end{definition}
Now, we are in a position to prove our main theorem.
\begin{proof}[Proof of Theorem \ref{main}]
First, we get a morphism $\hopfDG\to\hopfSF_+\to\hopfGP_+$ of Hopf monoids using Theorem \ref{sfgp} and Proposition \ref{pro:mor}. Let $g$ be a directed graph on the vertex set $I$.

The directed graph cone $\mathcal{C}(g)$ is a point if and only if $g$ has no edges. Therefore, when we restrict the basic character $\beta$ of $\hopfGP_+$ to directed graph cones, we obtain the basic character $\zeta$ of directed graphs. From Proposition \ref{prop:eqaapoly}, we have $\chi_g(n)=\chi_{\mathcal{C}(g)}(n)$, where $\chi_{\mathcal{C}(g)}(n)$ is the AA polynomial of the directed graph cone $\mathcal{C}(g)\in \hopfGP_+[I]$ obtained from the basic character $\beta$ of the Hopf monoid of extended generalized permutahedra. Using Proposition \ref{prop:generic}, it follows that $\chi_g(n)$ is the number of $\mathcal{C}(g)$-generic functions $y:I\to [n]$. Now, thanks to Theorem \ref{thm:polytope}, the normal fan to $\mathcal{C}(g)$ is a single cone cut out by inequalities $y(i)\le y(j)$ for the vertices $i,j\in I$ so that $g$ has a directed edge from $i$ to $j$. So the $\mathcal{C}(g)$-generic functions are the strictly order-reversing maps in $g$. We remark that there is a natural bijection between strictly order-reversing maps $I\to[n]$ and strictly order-preserving maps $I\to[n]$, and the proof is complete.
\end{proof}
Furthermore, this polynomial satisfies a reciprocity rule.
\begin{theorem}
Let $g$ be an acyclic directed graph with vertex set $I$ and $n \in \N$. If the polynomial $\chi_g(n)$ is the basic invariant for the Hopf monoid $\hopfDG[I]$, then we have
\[
(-1)^{|I|}\chi_g(-n)=\pi_g^{\geqslant}(n),
\]
where $\pi_g^{\geqslant}(n)$ is the weak-chromatic polynomial of $g$.
\end{theorem}
\begin{proof}
We will show the theorem using Proposition \ref{prop:reciprocity}. The directed graph cone $\mathcal{C}(g)$ has a vertex if and only if the directed graph $g$ has no directed cycles. If $y:I\to[n]$ is order-reversing, then there is a $y$-maximum face $\mathcal{C}(g)_y$, and it contains that single vertex. If $y$ is not order reversing, then $\mathcal{C}(g)$ is unbounded from above in the direction of $y$. So the left hand side, by Proposition \ref{prop:reciprocity}, is the number of order-reversing maps. These are bijective with order-preserving maps, whose number is the right hand side.

\end{proof}

\begin{corollary}
For any acyclic directed graph $g$ with vertex set $I$ and for any $n \in \N$, we have
\[
(-1)^{|I|}\pi^{>}_g(-n)=\pi_g^{\geqslant}(n).
\]
\end{corollary}

This corollary is equivalent to Lemma 6.5 in \cite{AB}. But we get the proof without invoking Ehrhart reciprocity.

\begin{example}
Let $g$ be the directed graph of Figure \ref{digex1}. The basic invariant $\chi_g(n)$  of $g$ is
\[
\chi_g(n)=\binom{n}{3}.
\]
\end{example}

\subsection{Edge character}\label{sec:edge}
Next, we introduce another character and its associated AA polynomial. This AA polynomial turns out to be a specialization of Awan--Bernardi's $B$-polynomial.

\begin{theorem}
For a directed graph $g=(I,E)\in\hopfDG[I]$, let the character $\eta$ be defined by $\eta(g)=q^{|E|}$.
Let $\psi_g(n)$ be the AA polynomial obtained from $\eta$.
For any directed graph $g=(I,E)\in\hopfDG[I]$, we have
\[
\psi_g(n)=q^{|E|}B_g(n,1/q,0).
\]
\end{theorem}
\begin{proof}
First, we may easily check that $\eta$ is a character of the Hopf monoid of directed graphs.

Let $I=S_1\sqcup \cdots \sqcup S_n$ be a decomposition. The coloring $f:I\to [n]$ is defined by $f(i)=k$ for $i\in S_k$. In this coloring, if $\Delta_{S_1\sqcup\cdots\sqcup S_n}(g)=0$, then there is an edge $(i,j)$ in $g$ such that $f(j)<f(i)$. If $\Delta_{S_1\sqcup\cdots\sqcup S_n}\ne0$, any edge of the directed graph $\mu(\Delta_{S_1\sqcup\cdots\sqcup S_n}(g))$ is an edge $(i,j)$ such that $f(i)=f(j)$. Furthermore, the edges of $g$ which did not remain in $\mu(\Delta_{S_1\sqcup\cdots\sqcup S_n}(g))$ are the edges $(i,j)$ such that $f(i)<f(j)$.

From these observations, we have
\begin{eqnarray*}
\psi_g(n)
&=&\sum_{f:I\to[n]\atop ^\nexists(i,j)\in E \mbox{{\tiny \ s.t.\ }} f(j)<f(i)}q^{|\{ (i,j)\in E \mid f(i)=f(j) \}|}\\
&=&q^{|E|}\sum_{f:I\to[n]\atop ^\nexists(i,j)\in E \mbox{{\tiny \ s.t.\ }} f(j)<f(i)}\left( \frac{1}{q} \right)^{|\{ (i,j)\in E \mid f(i)>f(j) \}|}\\
&=&q^{|E|} B_g(n,1/q,0).
\end{eqnarray*}

\end{proof}

\begin{example}
Let $g$ be the directed graph of Figure \ref{digex1}. We compute $\psi_g(n)$ which is the AA polynomial obtained from the edge character $\eta$:
\[
\psi_g(n)=q^3\binom{n}{1}+2q\binom{n}{2}+\binom{n}{3}.
\]
\end{example}

Finally, we get a reciprocity theorem from Theorem \ref{RPI}.

\begin{corollary}
Let $B_g(n,x,y)$ be Awan--Bernardi's $B$-polynomial of the directed graph $g$. Let $\antipode_I(g)$ be the antipode of $g=(I,E)$. Then we have
\[
q^{|E|} B_g(-1,1/q,0)=\zeta(\antipode_I(g)).
\]
More generally, for every $n\in \N$, we have
\[
B_g(-n,1/q,0)=B_{\antipode_I(g)}(n,q,0).
\]
where $B_{\antipode_I(g)}$ is defined by
\[
B_{\antipode_I(g)}= \sum_{I\scalebox{0.5}{\compo}(S_1,\ldots,S_k) \atop k\ge1}
(-1)^k B_{\mu_{S_1,\ldots,S_k}\circ\Delta_{S_1,\ldots,S_k}(g)}. 
\]
\end{corollary}

\end{document}